\documentclass[12pt]{article}
\usepackage{amssymb}
\usepackage{amsthm}
\usepackage{amsmath}
\usepackage{graphicx}
\usepackage{enumerate}
\usepackage{pict2e}
\usepackage{framed}

\DeclareMathOperator{\Max}{Max}
\DeclareMathOperator{\Min}{Min}

\newtheorem{theorem}{Theorem}[section]

\newtheorem{proposition}[theorem]{Proposition}

\newtheorem{remark}[theorem]{Remark}
\newtheorem{example}[theorem]{Example}

\title{Lagrange-like interpolation in unitary rings, Boolean algebras and Boolean posets}
\author{Ivan~Chajda and Helmut~L\"anger}
\date{}

\begin{document}

\footnotetext[1]{Support of the research of the first author by the Czech Science Foundation (GA\v CR), project 25-20013L, and by IGA, project P\v rF~2025~008, and support of the research of the second author by the Austrian Science Fund (FWF), project 10.55776/PIN5424624, is gratefully acknowledged.}

\footnotetext[2]{In fuzzy logic this function is called Baaz delta since the Austrian logician Matthias Baaz axiomatized it for G\"odel logics.}
	
\maketitle
	
\begin{abstract}
It is known that every function with a finite support over a given field can be interpolated by means of the Lagrangian polynomial. The question is if a similar interpolation is possible if one considers a unitary ring or a Boolean algebra instead of a field. We get a positive answer to this question provided the similarity type of the algebra in question is enriched with one more unary operation, the so-called Baaz delta. We get an explicit construction of this interpolation polynomial in both the cases. When going to Boolean posets, we have a lack of operations but these can be substituted by the operators $\Min U$ and $\Max L$. Hence, we generalize also the Baaz delta for posets as an operator and then we can derive an explicit interpolation term constructed by means of these operators also for Boolean posets.
\end{abstract}



\section{Introduction}

The aim of this short note is to show that a Lagrange-like interpolation of a given function with finite support by a certain polynomial is possible also in unitary rings and Boolean algebras when their signature is enriched with a certain unary operation, the so-called Baaz delta.

Having a field $\mathbf K=(K,+,\cdot)$, every function with a finite support, i.e.\ $f\colon\{a_1,\ldots,a_n\}\to K$ for some distinct elements $a_1,\ldots,a_n$ of $K$, can be interpolated by a polynomial $p(x)$ over $\mathbf K$ where $p(a_i)=f(a_i)$ for $i=1,\ldots,n$. The most profound interpolation is by means of the so-called Lagrangian polynomial, see e.g.\ \cite V. For the reader's convenience, we recall it.

In the following we use {\em Kronecker's delta} $\delta_{ij}$ defined by $\delta_{ij}:=1$ if $i=j$ and $\delta_{ij}:=0$ otherwise and the {\em Baaz delta} $^{2)}$ $\Delta$ defined by $\Delta(0):=0$ and $\Delta(x):=1$ otherwise.

\begin{theorem}\label{th4}
{\rm(Lagrange's interpolation formula)} Let $\mathbf K=(K,+,\cdot)$ be a field, $n>1$, $a_1,\ldots,a_n$ different elements of $K$ and $f\colon K\to K$ and define $p_1,\ldots,p_n,p\colon K\to K$ by
\begin{align*}
p_i(x) & \:=\prod_{\substack{j=1 \\ j\ne i}}^n\frac{x-a_j}{a_i-a_j}\text{ for }i=1,\ldots,n, \\
  p(x) & :=\sum_{i=1}^nf(a_i)p_i(x)
\end{align*}
for all $x\in K$. Then $p(x)$ is a polynomial over $\mathbf K$ satisfying $p(a_k)=f(a_k)$ for $k=1,\ldots,n$.
\end{theorem}

\begin{proof}
We have $p_i(a_k)=\delta_{ik}$ for $i,k=1,\ldots,n$ and hence
\[
p(a_k)=\sum_{i=1}^nf(a_i)p_i(a_k)=\sum_{i=1}^nf(a_i)\delta_{ik}=f(a_k)
\]
for $k=1,\ldots,n$.
\end{proof}

In the next sections we will modify the construction of the polynomials $p_i(x)$ in such a way that the interpolation polynomial $p(x)$ can be derived similarly as in Theorem~\ref{th4} also for unitary rings and Boolean algebras.

\section{Interpolation in unitary rings}

For unitary rings we can formulate our interpolation formula as follows:

\begin{theorem}\label{th3}
Let $(R,+,\cdot,0,1)$ be a unitary ring, $n>1$, $a_1,\ldots,a_n$ different elements of $R$ and $f\colon R\to R$ and define $p_1,\ldots,p_n,p\colon R\to R$ by
\begin{align*}
p_i(x) & :=\prod_{\substack{j=1 \\ j\ne i}}^n\Delta(x-a_j)\text{ for }i=1,\ldots,n, \\
  p(x) & :=\sum_{i=1}^nf(a_i)p_i(x)
\end{align*}
for all $x\in B$. Then $p(x)$ is a polynomial over the enriched unitary ring $(R,+,\cdot,\Delta,0,1)$ satisfying $p(a_k)=f(a_k)$ for $k=1,\ldots,n$.
\end{theorem}

\begin{proof}
It is evident that $p(x)$ is a polynomial over $(R,+,\cdot,\Delta,0,1)$. We have $p_i(a_k)=\delta_{ik}$ for $i,k=1,\ldots,n$ and hence
\[
p(a_k)=\sum_{i=1}^nf(a_i)p_i(a_k)=\sum_{i=1}^nf(a_i)\delta_{ik}=f(a_k)
\]
for $k=1,\ldots,n$.
\end{proof}

Note that the function $p(x)$ from Theorem~\ref{th3} is in fact a polynomial over the enriched unitary ring $(R,+,\cdot,\Delta,0,1)$, but it need not be a polynomial over $(R,+,\cdot,0,1)$.

\begin{example}
Consider a unitary ring $(R,+,\cdot,1)$, $a,b,c,d,e,g,h,i\in R$ such that $a,b,c,d$ are different and a mapping $f\colon R\to R$ with
\begin{align*}
f(a) & =e, \\
f(b) & =g, \\
f(c) & =h, \\
f(d) & =i.	
\end{align*}
In accordance with Theorem~\ref{th3} define $p_1,\ldots,p_4,p\colon R\to R$ by
\begin{align*}
p_1(x) & :=\Delta(x-b)\cdot\Delta(x-c)\cdot\Delta(x-d), \\	
p_2(x) & :=\Delta(x-a)\cdot\Delta(x-c)\cdot\Delta(x-d), \\	
p_3(x) & :=\Delta(x-a)\cdot\Delta(x-b)\cdot\Delta(x-d), \\	
p_4(x) & :=\Delta(x-a)\cdot\Delta(x-b)\cdot\Delta(x-c), \\
  p(x) & :=e\cdot p_1(x)+g\cdot p_2(x)+h\cdot p_3(x)+i\cdot p_4(x)
\end{align*}
for all $x\in R$. Then really
\begin{align*}
p(a) & =e, \\
p(b) & =g, \\
p(c) & =h, \\
p(d) & =i.	
\end{align*}
\end{example}

\section{Interpolation in Boolean algebras}

We continue with the case of Boolean algebras.

It is well-known that for Boolean algebras $(B,\vee,\wedge,{}',0,1)$ the corresponding Boolean ring $(B,+,\wedge,0,1)$ is an idempotent unitary ring satisfying the identity $x+x\approx0$ and hence also the identity $x-y\approx x+y$, see e.g. \cite B. The binary operation $+$ defined by
\[
x+y:=(x'\wedge y)\vee(x\wedge y')
\]
for all $x,y\in B$ is called {\em symmetric difference}. Now $x+y=0$ is equivalent to $x=y$ since $(B,+,0)$ is an involutory group. Hence we can derive the interpolation formula for Boolean algebras from that valid for unitary rings stated in Theorem~\ref{th3}.

\begin{theorem}\label{th1}
Let $(B,\vee,\wedge,{}',0,1)$ be a Boolean algebra, $n>1$, $a_1,\ldots,a_n$ different elements of $B$ and $f\colon B\to B$ and define $p_1,\ldots,p_n,p\colon B\to B$ by
\begin{align*}
p_i(x) & :=\bigwedge_{\substack{j=1 \\ j\ne i}}^n\Delta(x+a_j)\text{ for }i=1,\ldots,n, \\
  p(x) & :=\bigvee_{i=1}^n\big(f(a_i)\wedge p_i(x)\big)
\end{align*}
for all $x\in B$. Then $p(x)$ is a polynomial over the enriched Boolean algebra $(B,\vee,\wedge,{}',\Delta,0,1)$ satisfying $p(a_k)=f(a_k)$ for $k=1,\ldots,n$.
\end{theorem}

\begin{remark}
Instead of
\[
p(x):=\sum_{i=1}^n\big(f(a_i)\wedge p_i(x)\big)
\]
we can take
\[
p(x):=\bigvee_{i=1}^n\big(f(a_i)\wedge p_i(x)\big)
\]
since
\[
\bigvee_{i=1}^n\big(f(a_i)\wedge\delta_{ik}\big)=f(a_k)=\sum_{i=1}^n\big(f(a_i)\wedge\delta_{ik}\big)
\]
where the sum sign stands for the symmetric difference.
\end{remark}

Hence the function $p(x)$ from Theorem~\ref{th1} is indeed a polynomial over the enriched Boolean algebra $(B,\vee,\wedge,{}',\Delta,0,1)$.

\begin{example}
The lattice depicted in Fig.~1
	
\vspace*{-4mm}
	
\begin{center}
\setlength{\unitlength}{7mm}
\begin{picture}(12,10)
\put(6,1){\circle*{.3}}
\put(3,3){\circle*{.3}}
\put(5,3){\circle*{.3}}
\put(7,3){\circle*{.3}}
\put(9,3){\circle*{.3}}
\put(1,5){\circle*{.3}}
\put(3,5){\circle*{.3}}
\put(5,5){\circle*{.3}}
\put(7,5){\circle*{.3}}
\put(9,5){\circle*{.3}}
\put(11,5){\circle*{.3}}
\put(3,7){\circle*{.3}}
\put(5,7){\circle*{.3}}
\put(7,7){\circle*{.3}}
\put(9,7){\circle*{.3}}
\put(6,9){\circle*{.3}}
\put(6,1){\line(-3,2)3}
\put(6,1){\line(-1,2)1}
\put(6,1){\line(1,2)1}
\put(6,1){\line(3,2)3}
\put(3,3){\line(-1,1)2}
\put(3,3){\line(0,1)2}
\put(3,3){\line(1,1)2}
\put(5,3){\line(-2,1)4}
\put(5,3){\line(1,1)2}
\put(5,3){\line(2,1)4}
\put(7,3){\line(-2,1)4}
\put(7,3){\line(0,1)2}
\put(7,3){\line(2,1)4}
\put(9,3){\line(-2,1)4}
\put(9,3){\line(0,1)2}
\put(9,3){\line(1,1)2}
\put(1,5){\line(1,1)2}
\put(1,5){\line(2,1)4}
\put(3,5){\line(0,1)2}
\put(3,5){\line(2,1)4}
\put(5,5){\line(0,1)2}
\put(5,5){\line(1,1)2}
\put(7,5){\line(-2,1)4}
\put(7,5){\line(1,1)2}
\put(9,5){\line(-2,1)4}
\put(9,5){\line(0,1)2}
\put(11,5){\line(-2,1)4}
\put(11,5){\line(-1,1)2}
\put(3,7){\line(3,2)3}
\put(5,7){\line(1,2)1}
\put(7,7){\line(-1,2)1}
\put(9,7){\line(-3,2)3}
\put(5.85,.3){$0$}
\put(2.35,2.85){$a$}
\put(4.35,2.85){$b$}
\put(7.4,2.85){$c$}
\put(9.4,2.85){$d$}
\put(.35,4.85){$e$}
\put(2.35,4.85){$g$}
\put(4.35,4.85){$h$}
\put(7.4,4.85){$h'$}
\put(9.4,4.85){$g'$}
\put(11.4,4.85){$e'$}
\put(2.35,6.85){$d'$}
\put(4.35,6.85){$c'$}
\put(7.4,6.85){$b'$}
\put(9.4,6.85){$a'$}
\put(5.85,9.4){$1$}
\put(1,-.75){{\rm Figure~1. $16$-element Boolean algebra}}
\end{picture}
\end{center}
	
\vspace*{4mm}
	
is the $16$-element Boolean algebra $(B,\vee,\wedge,{}',0,1)$. Now consider a mapping $f\colon B\to B$ with
\begin{align*}
 f(a) & =c', \\
 f(g) & =a', \\
f(b') & =h, \\
 f(1) & =e'.	
\end{align*}
In accordance with Theorem~\ref{th1} define $p_1,\ldots,p_4,p\colon B\to B$ by
\begin{align*}
p_1(x) & :=\Delta(x+g)\wedge\Delta(x+b')\wedge\Delta(x+1), \\	
p_2(x) & :=\Delta(x+a)\wedge\Delta(x+b')\wedge\Delta(x+1), \\	
p_3(x) & :=\Delta(x+a)\wedge\Delta(x+g)\wedge\Delta(x+1), \\	
p_4(x) & :=\Delta(x+a)\wedge\Delta(x+g)\wedge\Delta(x+b'), \\
  p(x) & :=\big(c'\wedge p_1(x)\big)\vee\big(a'\wedge p_2(x)\big)\vee\big(h\wedge p_3(x)\big)\vee\big(e'\wedge p_4(x)\big)
\end{align*}
for all $x\in B$. Then really
\begin{align*}
 p(a) & =c', \\
 p(g) & =a', \\
p(b') & =h, \\
 p(1) & =e'.	
\end{align*}
\end{example}

\section{Interpolation in Boolean posets}

Now we turn our attention to Boolean posets (alias Boolean ordered sets, see e.g.\ \cite N and \cite P). For this purpose let us recall several concepts introduced already in our previous paper \cite{CKL}.

Let $\mathbf P=(P,\le)$ be a poset, $a,b\in P$ and $A,B\subseteq P$. Then {\em $\Max A$} and {\em $\Min A$} denote the set of all maximal and minimal elements of $A$, respectively. We define
\begin{align*}
A\le B & \text{ if }x\le y\text{ for all }x\in A\text{ and all }y\in B, \\
  L(A) & :=\{x\in P\mid x\le A\}, \\
  U(A) & :=\{x\in P\mid A\le x\}.
\end{align*}
Here and in the following we identify singletons with their unique element. Instead of $L(\{a\})$, $L(\{a,b\})$, $L(A\cup\{a\})$, $L(A\cup B)$ and $L\big(U(A)\big)$ we simply write $L(a)$, $L(a,b)$, $L(A,a)$, $L(A,B)$ and $LU(A)$. Analogously we proceed in similar cases. 

Hence, by the binary operator $\Max L(x,y)$ and $\Min U(x,y)$ we mean the mapping from $P^2$ to $2^P$ assigning to each $(x,y)\in P^2$ the set of all maximal elements in the lower cone $L(x,y)$ of $x$ and $y$ and the set of all minimal elements in the upper cone $U(x,y)$ of $x$ and $y$, respectively. Of course, if the poset in question is a lattice then $\Max L(x,y)=x\wedge y$ and $\Min U(x,y)=x\vee y$. Thus also in the more general case of a poset, $\Max L(x,y)$  and $\Min U(x,y)$ should play the role of $x\wedge y$ and $x\vee y$, respectively.

Moreover, if there exists $\sup(a,b)$ then $\Min U(a,b)=a\vee b$ and, similarly, if $\inf(a,b)$ exists then $\Max L(a,b)=a\wedge b$. Of course, if $(P,\le,0,1)$ is a bounded poset then $a\vee0=a$, $a\vee1=1$, $a\wedge0=0$ and $a\wedge1=a$.

Let $P$ be set and $\,'$ a unary operation on $P$. Then $\,'$ is called an {\em involution} if it satisfies the identity $(x')'\approx x$. If $(P,\le)$ is a poset then $\,'$ is called {\em antitone} if $x,y\in P$ and $x\le y$ imply $f(y)\le f(x)$. If $(P,\le,0,1)$ is a bounded poset then $\,'$ is called a {\em complementation} if it satisfies the identities $L(x,x')\approx0$ and $U(x,x')\approx1$. A {\em poset} $(P,\le)$ is called {\em distributive} (see e.g.\ \cite P) if it satisfies one of the following equivalent conditions:
\begin{align*}
 L\big(U(x,y),z\big) & \approx LU\big(L(x,z),L(y,z)\big), \\	
UL\big(U(x,y),z\big) & \approx U\big(L(x,z),L(y,z)\big), \\	
 U\big(L(x,y),z\big) & \approx UL\big(U(x,z),U(y,z)\big), \\
LU\big(L(x,y),z\big) & \approx L\big(U(x,z),U(y,z)\big).
\end{align*}
A {\em Boolean poset} is a distributive bounded poset with a complementation. It can be shown that the complementation of a Boolean poset is unique and it is an antitone involution (see e.g.\ \cite{CKL}).

For Boolean posets $(B,\vee,\wedge,{}',0,1)$ we define a binary operator $+\colon B^2\to2^B$ by
\[
x+y:=\Min U\big(L(x',y),L(x,y')\big)
\]
for all $x,y\in B$. Observe that $x+y=\Min U\big(\Max L(x',y),\Max L(x,y')\big)$ for all $x,y\in B$.

An elementary but important result for our study of interpolation is the following proposition which shows that the binary operator $+$ in Boolean posets satisfies a similar property as the symmetric difference in Boolean algebras.

\begin{proposition}\label{prop1}
Let $(B,\le,{}',0,1)$ be a Boolean poset and $a,b\in B$. Then $a+b=0$ if and only if $a=b$.
\end{proposition}

\begin{proof}
First assume $a+b=0$. Then $\Min U\big(L(a',b),L(a,b')\big)=0$ and hence $L(a',b)=L(a,b')=0$. Now we conclude
\begin{align*}
a' & \in U(a')=U(0,a')=U\big(L(a,b'),a'\big)=UL\big(U(a,a'),U(b',a')\big)=UL\big(1,U(a',b')\big)= \\
   & =ULU(a',b')=U(a',b')\subseteq U(b')
\end{align*}
and hence $b'\le a'$, i.e., $a\le b$. Interchanging the roles of $a$ and $b$ we obtain $b\le a$ and therefore $a=b$. If, conversely $a=b$ then
\[
a+b=a+a=\Min U\big(L(a',a),L(a,a')\big)=\Min U(0,0)=0.
\]
\end{proof}

It is worth noticing that $a+b$ need not be an element of $B$ but may be an arbitrary subset of $B$.

Although every finite uniquely complemented lattice is a Boolean algebra, see e.g.\ \cite D, this is not the case for posets.

The next example shows that the assertion from Proposition~\ref{prop1} need not hold if the complemented poset in question is not distributive.

\begin{example}
The poset $(P,\le,{}',0,1)$ visualized in Fig.~2
	
\vspace*{-4mm}
	
\begin{center}
\setlength{\unitlength}{7mm}
\begin{picture}(8,8)
\put(4,1){\circle*{.3}}
\put(1,3){\circle*{.3}}
\put(3,3){\circle*{.3}}
\put(5,3){\circle*{.3}}
\put(7,3){\circle*{.3}}
\put(1,5){\circle*{.3}}
\put(3,5){\circle*{.3}}
\put(5,5){\circle*{.3}}
\put(7,5){\circle*{.3}}
\put(4,7){\circle*{.3}}
\put(4,1){\line(-3,2)3}
\put(4,1){\line(-1,2)1}
\put(4,1){\line(1,2)1}
\put(4,1){\line(3,2)3}
\put(4,7){\line(-3,-2)3}
\put(4,7){\line(-1,-2)1}
\put(4,7){\line(1,-2)1}
\put(4,7){\line(3,-2)3}
\put(1,3){\line(0,1)2}
\put(1,3){\line(1,1)2}
\put(1,3){\line(2,1)4}
\put(3,3){\line(-1,1)2}
\put(3,3){\line(2,1)4}
\put(5,3){\line(-2,1)4}
\put(5,3){\line(1,1)2}
\put(7,3){\line(-2,1)4}
\put(7,3){\line(-1,1)2}
\put(7,3){\line(0,1)2}
\put(3.85,.3){$0$}
\put(.35,2.85){$a$}
\put(2.35,2.85){$b$}
\put(5.4,2.85){$c$}
\put(7.4,2.85){$d$}
\put(.35,4.85){$d'$}
\put(2.35,4.85){$c'$}
\put(5.4,4.85){$b'$}
\put(7.4,4.85){$a'$}
\put(3.85,7.4){$1$}
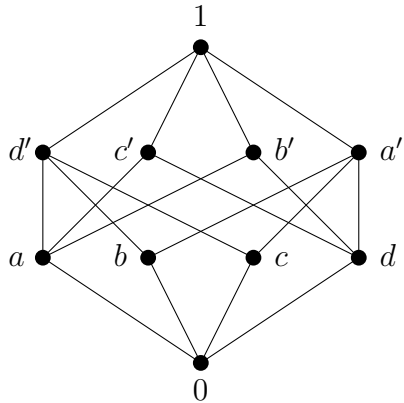
\put(-4.8,-.75){{\rm Figure~2. Non-distributive bounded poset with a complementation}}
\end{picture}
\end{center}
	
\vspace*{4mm}
	
is not distributive since
\[
L\big(U(a,b),c\big)=L(d',1,c)=L(c)\ne0=LU(0,0)=LU\big(L(a,c),L(b,c)\big),
\]
but the unary operation $\,'$ defined by the table
\[
\begin{array}{l|llllllllll}
x  & 0 & a  & b  & c  & d  & a' & b' & c' & d' & 1 \\
\hline
x' & 1 & a' & b' & c' & d' & a  & b  & c  & d  & 0
\end{array}
\]
is a complementation on it. We have
\[
b+c=\Min U\big(L(b',c),L(b,c')\big)=\Min U(0,0)=\min U(0)=0
\]
though $b\ne c$.
\end{example}

The concept of Baaz delta can be extended for a poset $(P,\le,0,1)$ as follows: We define $\Delta(0):= 0$ and $\Delta(A):=1$ for each subset $A$ of $P$ different from $\{0\}$. In this case it is not an operation but a unary operator on the set $P$. A term constructed by the afore mentioned operators $\Max L$, $\Min U$, ${}'$ and $\Delta$ will be called an {\em operator term} over $\mathbf P=(P,\le,{}',\Delta,0,1)$. In particular, also the operator $+$ as defined above is an operator term over $\mathbf P$.

For a Boolean poset enriched with the operator $\Delta$ we can state and prove the following interpolation formula.

\begin{theorem}\label{th2}
Let $(B,\le,{}',0,1)$ be a Boolean poset, $n>1$, $a_1,\ldots,a_n$ different elements of $B$ and $f\colon B\to B$ and define $p_1,\ldots,p_n\colon B\to B$ and $p\colon B\to2^B$ by
\begin{align*}
p_i(x) & :=\bigwedge_{\substack{j=1 \\ j\ne i}}^n\Delta(x+a_j)\text{ for }i=1,\ldots,n, \\
  p(x) & :=\Min U\Big(\bigcup_{i=1}^n\big(f(a_i)\wedge p_i(x)\big)\Big)
\end{align*}
for all $x\in B$. Then $p(x)$ is an operator term over $(B,\le,{}',\Delta,0,1)$ satisfying $p(a_k)=f(a_k)$ for $k=1,\ldots,n$.
\end{theorem}

\begin{remark}
Observe that
\[
\bigwedge_{\substack{j=1 \\ j\ne i}}^n\Delta(x+a_j)
\]
can be written in the form
\[
\Max L\big(\bigcup_{\substack{j=1 \\ j\ne i}}^n\Delta(x+a_j)\big)
\]
and $f(a_i)\wedge p_i(x)$ in the form $\Max L\big(f(a_i),p_i(x)\big)$.
\end{remark}

\begin{proof}[Proof of Theorem~\ref{th2}]
Clearly,
\[
\bigwedge_{\substack{j=1 \\ j\ne i}}^n\Delta(x+a_j)
\]
is defined because every member $\Delta(x+a_j)$ of this intersection equals either $0$ or $1$. Thus we have $p_i(a_k)=\delta_{ik}$ for $i,k=1,\ldots,n$ and hence
\[
p(a_k)=\Min U\Big(\bigcup_{i=1}^n\big(f(a_i)\wedge p_i(a_k)\big)\Big)=\Min U\Big(\bigcup_{i=1}^n\big(f(a_i)\wedge\delta_{ik}\big)\Big)=\Min U\big(f(a_k)\big)=f(a_k)
\]
for $k=1,\ldots,n$.
\end{proof}

That the interpolation formula from Theorem~\ref{th2} really works for Boolean posets which need not be Boolean algebras is shown in the following example.

\begin{example}\label{ex1}
Consider the Boolean poset $(B,\le,0,1)$ depicted in Fig.~3
	
\vspace*{-4mm}
	
\begin{center}
\setlength{\unitlength}{7mm}
\begin{picture}(8,8)
\put(4,1){\circle*{.3}}
\put(1,3){\circle*{.3}}
\put(3,3){\circle*{.3}}
\put(5,3){\circle*{.3}}
\put(7,3){\circle*{.3}}
\put(1,5){\circle*{.3}}
\put(3,5){\circle*{.3}}
\put(5,5){\circle*{.3}}
\put(7,5){\circle*{.3}}
\put(4,7){\circle*{.3}}
\put(4,1){\line(-3,2)3}
\put(4,1){\line(-1,2)1}
\put(4,1){\line(1,2)1}
\put(4,1){\line(3,2)3}
\put(4,7){\line(-3,-2)3}
\put(4,7){\line(-1,-2)1}
\put(4,7){\line(1,-2)1}
\put(4,7){\line(3,-2)3}
\put(1,3){\line(0,1)2}
\put(1,3){\line(1,1)2}
\put(1,3){\line(2,1)4}
\put(3,3){\line(-1,1)2}
\put(3,3){\line(0,1)2}
\put(3,3){\line(2,1)4}
\put(5,3){\line(-2,1)4}
\put(5,3){\line(0,1)2}
\put(5,3){\line(1,1)2}
\put(7,3){\line(-2,1)4}
\put(7,3){\line(-1,1)2}
\put(7,3){\line(0,1)2}
\put(3.85,.3){$0$}
\put(.35,2.85){$a$}
\put(2.35,2.85){$b$}
\put(5.4,2.85){$c$}
\put(7.4,2.85){$d$}
\put(.35,4.85){$d'$}
\put(2.35,4.85){$c'$}
\put(5.4,4.85){$b'$}
\put(7.4,4.85){$a'$}
\put(3.85,7.4){$1$}
\put(.8,-.75){{\rm Figure~3. Boolean poset}}
\end{picture}
\end{center}
	
\vspace*{4mm}
	
One can see that e.g.\ $b+c'$ is not an element of $B$, namely
\[
b+c'=U\big(L(b',c'),L(b,c)\big)=\Min U(a,d)=\{b',c'\}.
\]
Now consider a mapping $f\colon B\to B$ with
\begin{align*}
 f(0) & =a, \\
 f(a) & =c, \\
 f(b) & =d', \\
f(c') & =1.
\end{align*}
In accordance with Theorem~\ref{th2} define $p_1,\ldots,p_4\colon B\to B$ and $p\colon B\to2^B$ by
\begin{align*}
p_1(x) & :=\Delta(x+a)\wedge\Delta(x+b)\wedge\Delta(x+c'), \\	
p_2(x) & :=\Delta(x+0)\wedge\Delta(x+b)\wedge\Delta(x+c'), \\	
p_3(x) & :=\Delta(x+0)\wedge\Delta(x+a)\wedge\Delta(x+c'), \\	
p_4(x) & :=\Delta(x+0)\wedge\Delta(x+a)\wedge\Delta(x+b), \\
  p(x) & :=\Min U(\{a\wedge p_1(x),c\wedge p_2(x),d'\wedge p_3(x),1\wedge p_4(x)\})
\end{align*}
for all $x\in B$. Then really
\begin{align*}
 p(0) & =a, \\
 p(a) & =c, \\
 p(b) & =d', \\
p(c') & =1.
\end{align*}
\end{example}

\begin{remark}
	Having in mind the non-distributive complemented poset from Example~\ref{ex1}, we see immediately that a similar interpolation formula is not possible in this case since $p_i(a_k)$ need not be equal to Kronecker's delta $\delta_{ik}$, namely, if $n=2$, $a_1=b$, $a_2=c$ and $f(a_1)\ne0$ then
	\begin{align*}
		p_1(a_1) & =\Delta(b+c)=\Delta(0)=0\ne1=\delta_{11}, \\
		p_2(a_1) & =\Delta(b+b)=\Delta(0)=0, \\
		p(a_1) & =\Min U\big(f(a_1)\wedge p_1(a_1),f(a_2)\wedge p_2(a_1)\big)=\Min U(0,0)=0\ne f(a_1).
	\end{align*}
\end{remark}

%\section{Declarations}

%{\bf Compliance with Ethical Standards} This article does not contain any studies with human participants or animals performed by any of the authors.

%{\bf Funding} Support of the research of the first author by the Czech Science Foundation (GA\v CR), project 24-14386L, entitled ``Representation of algebraic semantics for substructural logics'', and by IGA, project P\v rF~2024~011, is gratefully acknowledged.

%{\bf Data availability statement} No datasets were generated or analyzed during the current study.

%{\bf Competing interests} There are no competing interests of a financial or personal nature between the authors.

%{\bf Conflict of interest} Both authors declare that they have no conflict of interest.

%{\bf Authors' contributions} Both authors contributed equally to the manuscript.

Authors' addresses:

Ivan Chajda \\
Palack\'y University Olomouc \\
Faculty of Science \\
Department of Algebra and Geometry \\
17.\ listopadu 12 \\
771 46 Olomouc \\
Czech Republic \\
ivan.chajda@upol.cz

Helmut L\"anger \\
TU Wien \\
Faculty of Mathematics and Geoinformation \\
Institute of Discrete Mathematics and Geometry \\
Wiedner Hauptstra\ss e 8-10 \\
1040 Vienna \\
Austria, and \\
Palack\'y University Olomouc \\
Faculty of Science \\
Department of Algebra and Geometry \\
17.\ listopadu 12 \\
771 46 Olomouc \\
Czech Republic \\
helmut.laenger@tuwien.ac.at

\begin{thebibliography}9
\bibitem B
G.~Birkhoff, Lattice Theory. AMS, Providence, RI, 1979. ISBN 0-8212-1025-1.
\bibitem{CKL}
I.~Chajda, M.~Kola\v r\'ik and H.~L\"anger, Operators $\Max L$ and $\Min U$ and duals of Boolean posets. J.~Multiple-Valued Logic Soft Computing (submitted).
\bibitem D
R.~P.~Dilworth, Lattices with unique complements, Trans.\ Amer.\ Math.\ Soc.\ {\bf57} (1945), 123--154.
\bibitem N
J.~Niederle, Boolean and distributive ordered sets: characterization and representation by sets. Order {\bf12} (1995), 189--210.
\bibitem P
J.~Paseka, Note on distributive posets, Math.\ Appl.\ {\bf1} (2012), 197--206.
\bibitem V
B.~L.~van~der~Waerden, Algebra.\ Vol.\ I. Springer, New York 1991. ISBN 0-387-97424-5.
\end{thebibliography}
\end{document}